\documentclass[a4paper,10pt]{amsart}
\usepackage{amssymb}
\usepackage{graphicx} 
\usepackage{color}
\usepackage{mathrsfs}
\usepackage{txfonts}
\usepackage{soul}

\newcounter{contai}

\newtheorem{theorem}{Theorem}

\newtheorem{definition}{Definition}
\newtheorem{lemma}[theorem]{Lemma}
\newtheorem{proposition}[theorem]{Proposition} 
\newtheorem{remark}{Remark} 
\newcommand{\NN}{{\rm\bf N}}

\newcommand{\ZZ}{{\rm\bf Z}}
\newcommand{\RR}{{\rm\bf R}}

\newcommand{\dpt}{\displaystyle}

\begin{document}

\title[Is there switching without suspended horseshoes?]{Is there switching without suspended horseshoes?}


\author[Alexandre A. P. Rodrigues]{Alexandre A. P. Rodrigues}
\address[A.A.P. Rodrigues]{Centro de Matem\'atica da Universidade do Porto\\ 
and Faculdade de Ci\^encias da Universidade do Porto\\
Rua do Campo Alegre 687, 4169--007 Porto, Portugal}

\email[Alexandre Rodrigues]{alexandre.rodrigues@fc.up.pt}

\thanks{CMUP is supported by the European Regional Development Fund through the programme COMPETE and by the Portuguese Government through the Funda\c{c}\~ao para a Ci\^encia e a Tecnologia (FCT) under the project PEst-C/MAT/UI0144/2011.  AR was supported by the grant SFRH/BPD/84709/2012 of FCT. Part of this work has been written during the author stays in Nizhny Novgorod University supported by the grant RNF 14-41-00044}

\begin{abstract}
In general, infinite switching behaviour near networks is associated with the existence of suspended horseshoes. Trajectories that realize switching lie within these transitive sets. In this note, revisiting the equivariant Shilnikov scenario, we describe an attracting homoclinic network exhibiting forward switching and without suspended horseshoes in its neighbourhood.  Thus we provide an example of an asymptotically stable network exhibiting sensitive dependence on initial conditions. 
\medbreak
{\sc Resum\'e.}
En g\'en\'eral, le comportement de commutation autour des r\'eseaux  est associ\'e \`a l'existence de fers \`a cheval en suspension. Trajectoires qui r\'ealisent la commutation se situent dans ces ensembles transitifs. Dans cette note, visitant le sc\'enario classique de Shilnikov avec une sym\'etrie, nous d\'ecrivons un r\'eseau homocline qui pr\'esente commutation n'ayant pas des fers \`a cheval en suspension autour de lui. Plus pr\'ecisement nous donnons un exemple d'un r\'eseau asymptotiquement stable avec sensible d\'ependance des conditions initiales.
\end{abstract}

\maketitle

\section{Introduction}

The existence of homo and heteroclinic cycles in systems with symmetry is no longer a surprising feature. There are several examples of cycles arising in differential equations symmetric under the action of a specific compact Lie group.  Several definitions of heteroclinic cycles and networks have been given in the literature. 
These objects are associated with intermittent dynamics and used to model stop-and-go behaviour in various applications. Throughout the present article, we use the following definition valid for a finite dimensional system of ordinary equations (ODE):
\begin{definition}
A \emph{heteroclinic cycle} is a finite collection of equilibria $\{O_1,\ldots, O_n\}$ of the ODE together with a set of heteroclinic connections $\{\gamma_1,\ldots, \gamma_n\}$ where $\gamma_j$ is a solution of the ODE such that:
$$
\lim_{t \rightarrow -\infty}\gamma_j=O_j \qquad \text{and} \qquad  \lim_{t \rightarrow +\infty}\gamma_j=O_{j+1}
$$
and $O_{n+1} \equiv O_1$. When $n=1$, we say that the set $\{O_1, \gamma_1\}$ is a \emph{homoclinic cycle}. A \emph{homoclinic network} is a connected finite union of homoclinic cycles associated to the same equilibrium point. 
\end{definition}

\medbreak
 
Complex behaviour near a homo/heteroclinic network is often connected to the occurence of \emph{switching}. There are different types of switching, leading to increasingly complex behaviour near the network:  
\begin{itemize}
\item \emph{switching at a node}  \cite{Aguiar_games} characterised by existence of initial conditions, near an incoming connection to that node, whose trajectory follows any of the possible outgoing connections at the same node.  Incoming connection does not predetermine the outgoing choice at the node.
\medbreak
\item \emph{switching along a heteroclinic connection} \cite{ACL NONLINEARITY, CL}, which extends the notion of switching at a node to initial conditions whose trajectories follow a prescribed homo/heteroclinic connection.
\medbreak

\item \emph{infinite switching} \cite{ALR, HK, Rodrigues3}, which ensures that any sequence of connections in the network is a possible path near the network. This is different from \emph{random switching} in which trajectories shadow the network in a non-controllable  way \cite{PD}. 
\end{itemize}

The absence of \emph{switching along a connection} prevents \emph{infinite switching} and, therefore, chaotic behaviour near the network. The term \emph{switching} has also been used to describe simpler dynamics where there is one change in the choices observed in trajectories. This is the case described in \cite{KS}. In this case, the network consists of two cycles and trajectories are allowed to change from a neighbourhood of one cycle to a neighbourhood of the other cycle. This change is referred to as switching, although it is a very weak example of this phenomenon. In \cite{CLP},  the expression \emph{railroad switching} is used in relation to switching at a node. Complex behaviour near a network can also arise from the presence of noise-induced switching, see \cite{ASK}. We do not address the presence of noise in this note.

The authors of \cite{PD}  find a form of complicated switching (possibly not infinite)  leading to regular and irregular cycling near a network. There are several examples in the literature where the existence of infinite switching leads to chaotic behaviour near the network, see \cite{ALR, LR, Rodrigues2, LR3}. All the networks considered by these authors have at least one invariant saddle at which the linearized vector field has complex eigenvalues. In general, infinite switching is related with the existence of suspended horseshoes in its neighbourhood. See for example the works \cite{IR, LR, Rodrigues}. In these articles, the authors proved the existence of infinitely many initial conditions that realize a given forward infinite path. These solutions  lie on the sequence of suspended horseshoes that accumulate on the network. So, infinite switching seems to be connected with the existence of suspended horseshoes. The natural question is:
\medbreak
\textbf{(Q1)} are there examples of homo/heteroclinic networks exhibiting infinite switching and without suspended horseshoes around it?

\medbreak
 The main goal of this note is to answer this question. We will exhibit a class of vector fields whose flow has a homoclinic network exhibiting infinite switching and without suspended horseshoes around it. The example is based on the most famous and rich examples in the dynamical systems theory  -- the Shilnikov model of a homoclinic cycle to a saddle-focus with negative saddle value \cite{Shilnikov65, Shilnikov_67, TS}. Although we deal with the classic Shilnikov homoclinic loop, as far as we know, the combination \emph{$\ZZ_2$-symmetry} and \emph{negative saddle-value} of the saddle-focus is new. This note provides an example of a simple attracting network exhibiting sensitivity with respect to initial conditions.

\section{Preliminaries}
\label{Preliminaries}
Let $M$ be a compact three-dimensional manifold possibly without boundary and let $\mathcal{X}^r(M)$ the Banach space of $C^r$ vector fields on $M$ endowed with the $C^r$ Whitney topology with $r\geq 2$. 
Consider a vector field $f: M \rightarrow TM$ defining a system
 \begin{equation}
    \label{general}
\dot{x}=f(x), \qquad x(0)=x_0\in M
 \end{equation}
and denote by $\varphi(t,x_0)$, with $t \in \RR$, the associated flow (with initial condition $x_0$). 

\subsection{Homoclinic Network}
In this paper, we will be focused on an equilibrium $O$ of (\ref{general}) such that its spectrum (\emph{i.e.} the eigenvalues of $df|_O$) consists of one pair of non-real complex numbers with negative real part and one positive real eigenvalue. This is what we call a \emph{saddle-focus}.  A \emph{homoclinic connection} associated to $O$ is a trajectory biasymptotic to $O$ in forward and backward times. A \emph{homoclinic network} associated to $O$ consists of the equilibrium and a finite union of homoclinic connections associated to O. 

\subsection{Homoclinic switching}
Let $\Gamma=\{O\}\cup \left(\bigcup_{i=1}^m \gamma_i\right)$ be a homoclinic network associated to $O$ with $m\in \mathbf{N}\backslash \{1\}$.

\begin{definition}
\label{Def2}
If $k\in\mathbf{N}$, a finite path of order $k$ on $\Gamma$ is a sequence
     $$
     \left(\gamma_{\sigma(1)},\ldots,\gamma_{\sigma(k)}\right) \in \{\gamma_1,\ldots,\gamma_m\}^k
     $$
     of homoclinic connections, where $\sigma:\{1,\ldots,k\}\rightarrow \{1,\ldots,m\}$ is an arbitrary map. We use the notation $\sigma^{k}$ for this type of finite path. For an infinite path, take $\sigma:\mathbf{N}\rightarrow \{1,\ldots,m\}$.
\end{definition}

Let $N_{\Gamma}$ be a neighbourhood of the network $\Gamma$ and let $V_{O}\subset N_{\Gamma}$ be a neighbourhood of  $O$. For each homoclinic connection $\gamma_i$ in $\Gamma$, consider a point $p_i\in \gamma_i$ and a neighbourhood $V_i\subset N_{\Gamma}$ of $p_i$. The collection of neighbourhoods $\{V_O,V_1,\ldots,V_m\}$ should be pairwise disjoint.

\begin{definition}
Given neighbourhoods as above,  we say that the trajectory of a point $q$ follows a finite path
$\sigma^k$, if there exist two monotonically increasing sequences of times $(t_{j})_{j\in \{1,\ldots,k+1\}}$ and $(z_{j})_{j\in \{1,\ldots,k\}}$ such that for all $j \in \{1,\ldots,k\}$, we have $t_{j}<z_{j}<t_{j+1}$ and:

\begin{enumerate}
\item[(i)]
$\varphi (t,q)\subset N_{\Gamma}$ for all $t\in ]t_{1},t_{k+1}[$;
\item[(ii)]
$\varphi (t_{j},q) \in V_O$ for all $j \in \{1,\ldots,k+1\}$ and $\varphi (z_{j},q)\in V_{\sigma(j)}$ for all $j \in \{1,\ldots,k\}$;
\item[(iii)] for all $j=1,\ldots,k-1$ there exists a proper subinterval $I\subset ]z_{j},z_{j+1}[$ such that, given $t\in ]z_{j},z_{j+1}[$, $\varphi(t,q)\in V_O$ if and only if $t\in I$.
\end{enumerate}
\end{definition}
The notion of a trajectory following an infinite path can be stated similarly. Along the paper, when we refer to points that follow a path, we mean that their trajectories do it.  Based in  \cite[\S 2]{ALR}, we define:

\begin{definition} There is:
\label{Def1}
\begin{enumerate}
\item[(i)]  \emph{finite switching} near $\Gamma$ if  for each finite path and for each neighbourhood $N_{\Gamma}$ there is a trajectory in $N_{\Gamma}$ that follows it and
\item[(ii)]  \emph{infinite switching} (or simply \emph{switching}) near $\Gamma$  by requiring that for each infinite path and for each neighbourhood $N_{\Gamma}$ there is a trajectory in $N_{\Gamma}$ that follows it.
\end{enumerate}
\end{definition}

An infinite path on $\Gamma$ can be considered as a pseudo-orbit of (\ref{general}) with infinitely many discontinuities. Switching near $\Gamma$ means that any pseudo-orbit in $\Gamma$ can be realized. In \cite{HK}, using connectivity matrices, the authors gave an equivalent definition of switching, emphasising the possibility of coding all trajectories that remain in a given neighbourhood of the network in both finite and infinite times.

\section{Main result}
Our object of study is the dynamics around a special type of homoclinic network, for which we give a rigorous description here. Specifically we consider  a family of vector fields in $\mathcal{X}^r(M)$, $r\geq 2$, with a flow given by the unique solution $\varphi(t,x) \in M$ of (\ref{general}) satisfying the following hypotheses:
\medbreak
\begin{enumerate}
\item[\textbf{(H1)}] The point $O$ is a saddle-focus equilibrium where the eigenvalues of $\mathrm{d}f|_{O}$ are $-C \pm \alpha i$ and  $E$, where $C, E, \alpha \in \RR^+$ and $C>E$.
\medbreak
\item[\textbf{(H2)}]  There is a trajectory $\gamma_1$ biasymptotic to $O$.
\medbreak
\item[\textbf{(H3)}] The vector field is $\ZZ_2$-symmetric under the action of $-Id$.
\medbreak
\end{enumerate}

Hypotheses \textbf{(H2)} and \textbf{(H3)} imply the presence of an additional trajectory, say $\gamma_2=-Id(\gamma_1)$. Thus $\Gamma=\{O\}\cup \gamma_1\cup \gamma_2$ is a homoclinic network; in particular $m=2$ in Definition \ref{Def2}.  We address the reader to \cite{Golubitsky} for more details about equations with symmetry. 
\medbreak
There are several papers in the literature dealing with the case where the inequality $C>E$ fails, all of them dealing with the complexity of solutions in a neighbourhood of $\Gamma$, namely a sequence of suspended horseshoes accumulating on the network. Finitely many of these horseshoes survive under the addition of generic perturbing terms.  Our main result says that although $\Gamma$ is attracting (the statistical limit set associated to $\Gamma$ is the point $O$), the approach to  the network is chaotic. 

\begin{theorem}
\label{Main}
For a vector field $f: M \rightarrow TM$ whose flow satisfies \textbf{(H1)}--\textbf{(H3)}, the following conditions hold:
\begin{itemize}
\item[\textbf{(a)}] there are no suspended horseshoes in the neighbourhood of $\Gamma$;
\item[\textbf{(b)}] the network $\Gamma$ is asymptotically stable, in the sense that all trajectories starting in a small open neighbourhood of $\Gamma$ are attracted to the network;
\item[\textbf{(c)}] there is infinite switching near $\Gamma$, realized by infinitely many initial conditions.
\end{itemize}
 \end{theorem}
 
 The proof of \textbf{(a)} and \textbf{(b)} may be found in the works \cite{GS1, Shilnikov65}. See also \cite{hom, TS}. The proof of \textbf{(c)} is addressed in \S \ref{Sec6} of the present note. In particular, in a small neighbourhood of $\Gamma$, $N_\Gamma$ , if $g$ is $C^2$-close to $f$, the set of non-wandering trajectories of $f$ in $N_\Gamma$ consists of $O$ and one or two attracting limit cycles. Using Theorem \ref{Main}, we conclude that the transient  dynamics should visit the ghost of the homoclinic cycles (in any prescribed order) before falling on the basins of attractions of the periodic solutions.
 
 \medbreak
 \subsection*{Structure of this Note:} In \S\ref{localdyn} we linearize the vector field around the saddle-focus, obtaining an isolating block around it; this section is concerned with introducing the notation for the proof of switching. In \S \ref{global}, we obtain a geometrical description of the way the flow transforms a segment of initial conditions across the stable manifold of $O$. This curve is wrapped around the isolating block and accumulates on the unstable manifold of $O$ ($\lambda$-Lemma for flows) and, in particular, on the next connection. The local stable manifold of $O$ crosses infinitely many times the previous curve. The geometric setting is explored in \S \ref{Sec6} to obtain intervals of the segment that are mapped by the flow into curves next to $O$ in a position similar to the first one. This allows to establish the recurrence neeeded for infinite switching. For any infinite sequence of homoclinic connections, say:
 $$
 \gamma_2 \rightarrow \gamma_1 \rightarrow \gamma_1 \rightarrow \gamma_1 \rightarrow \gamma_2 \rightarrow \gamma_2 \rightarrow \ldots
 $$
 without using perturbation theory, we find infinitely many trajectories that visits the neighbourhoods of these connections in the same sequence. Throughout this note, we have endeavoured to make a self
contained exposition bringing together all topics related to the proofs. We have stated short lemmas and we have drawn illustrative figures to make the paper easily readable.

\section{Local Maps}
\label{localdyn}

The behaviour of the vector field $f$ in the neighbourhood of the network $\Gamma$ is given, up
to topological equivalence, by the linear part of $f$ in the neighbourhood of $O$ and by
the transition map between two discs transversal to the flow in those neighbourhoods. In this section, we
choose coordinates in the neighbourhood of $O$ in order to put $f$ in the canonical form
and we assume that the transition map is linear.
The main point is the application of Samovol's Theorem \cite{Samovol} to $C^1$--linearize the flow around $O$, and to introduce cylindrical coordinates around the equilibrium. There are no $C^1$-resonances here. 
We use neighbourhoods with boundary transverse to the linearized flow.

\subsection{$C^1$--linearization}
\label{free}
Since $O$ is hyperbolic, by Samovol's Theorem \cite{Samovol}, the vector field $f$ is $C^{1}$--conjugate to its linear part in a $\varepsilon$-small open neighbouhood around $O$, $\varepsilon>0$. We choose cylindrical coordinates $(\rho ,\theta ,z)$
near $O$ so that the linearized vector field can be written as:
\begin{equation}
\label{local map O}
\left\{ 
\begin{array}{l}
\dot{\rho}=-C\rho \\
 \dot{\theta}=\alpha \\
 \dot{z}=E z
\end{array} .
\right.
\end{equation}

\begin{figure}[ht]
\begin{center}
\includegraphics[height=7cm]{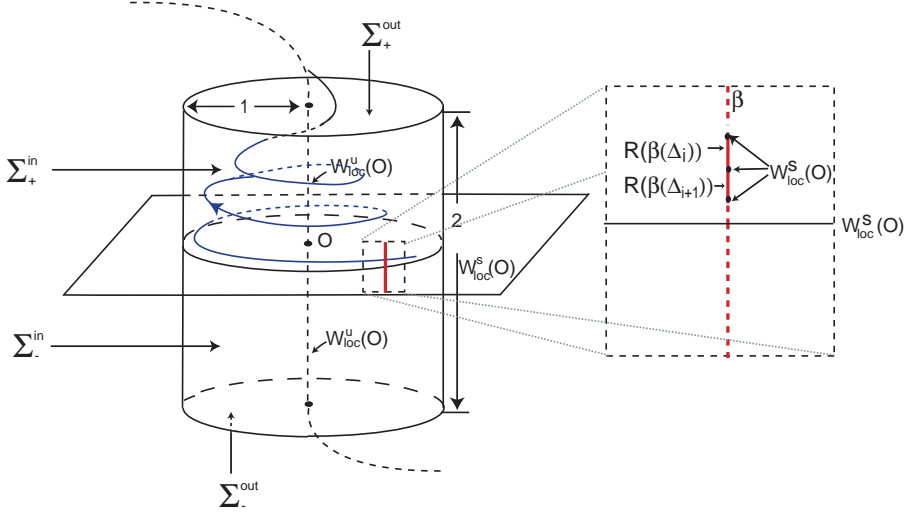}
\end{center}
\caption{\small  Cylindrical neighbourhood of the saddle-focus $O$. For $i,j\in\{-,+\}$, on a segment $\beta$ there are infinitely many subsegments that are mapped by $R$ into $\Sigma^{in}_i$, each one containing a point mapped into $W^s_{loc}(O)$. The small sub-segments contain smaller ones that are mapped into by $R^2$ into $\Sigma^{in}_j$ and the process may be continued forming a nested sequence. }
\label{neigh_vw}
\end{figure}

After a linear rescaling of the local variables, we consider a cylindrical neighbourhood of $O$ of radius $1$ and height $2$ that we denote by $V$ -- see Figure \ref{neigh_vw}. Their boundaries consist of three components: the cylinder wall parametrized by $x\in \RR\pmod{2\pi}$ and $|y|\leq 1$ with the usual cover $ (x,y)\mapsto (1 ,x,y)=(\rho ,\theta ,z)$ and two disks (top and bottom). We take polar coverings of these disks $(r,\phi )\mapsto (r,\phi , j)=(\rho ,\theta ,z)$ 
where $j\in \{-1,+1\}$, $0\leq r\leq 1$ and $\phi \in \RR\pmod{2\pi}$. By convention, the intersection points of $\Gamma$ with the wall of the cylinder has $0$ and $\pi$ angular coordinate. The set of points in the cylindrical wall with positive (resp. negative) second coordinate  is denoted by $\Sigma^{in}_+$. (resp $\Sigma^{in}_-$). The similar holds for $\Sigma^{out}$. 
\medbreak
As depicted in Figure \ref{neigh_vw}, the cylinder wall of $V$ is denoted by $\Sigma^{in}$.  Note that $W^s_{loc}(O)$ corresponds to the circle $ y=0$. The top and the bottom of
the cylinder are simply denoted  by $\Sigma^{out}$. The boundary of $V$ can be written as the disjoint union: $$\partial V  =   \Sigma^{in} \dot{\cup} \Sigma^{out}  \dot{\cup}  \Omega, $$ where $\Omega$ is the part of $\partial V$ where the flow is not transverse.
It follows by the above construction that:

\begin{lemma} Let $j\in \{-,+\}$. Solutions starting:
\begin{enumerate}
\item at $\Sigma^{in}$ go inside the cylinder $V$ in positive time;
\item at $\Sigma^{out}$ go outside the cylinder $V$ in positive time;
\item at $\Sigma_j^{in} \backslash W^s(O)$ leave the cylindrical neighbourhood $V$ at $\Sigma_j^{out}$.
\end{enumerate}
\end{lemma}

If $(x,y) \in \Sigma^{in}\backslash W^s_{loc}(O)$, let $T(x,y)$ be the time of flight through $V$ of the trajectory whose initial condition is $(x,y)$. It only depends on $y\neq 0$ and is given explicitly by 
\begin{equation}
\label{time_v}
T(x,y)=\frac{1}{E} \ln \left(\frac{1}{|y|}\right).
\end{equation}
In particular $\lim_{y\rightarrow 0} T(x,y)= + \infty$. Now, we obtain the expression of the local map that sends points in the boundary where the flow goes in, into points in the boundary where the flows goes out. The local map
$\Phi_{O }:\Sigma^{in} \rightarrow \Sigma^{out}$  near $O$ is given by
\begin{equation}
\begin{array}{c}
\Phi_{O }(x,y)=\left(y^{\delta},-\frac{\alpha}{E}\ln |y|+x\right)=(r,\phi)
\end{array}
\label{local_v}
\end{equation}
where $\delta=\frac{C_{ }}{E_{}} >1$ is the \emph{saddle index} of $O$. Observe that if $x_0 \in \RR$ is fixed, then $$\lim_{y\rightarrow 0} |\Phi_{O }(x_0,y)|= (0, + \infty).$$

The same expression holds  for the local map from the other connected component of $\Sigma^{in}_- \backslash W^s_{loc}(O )$ to $\Sigma^{out}_-$ (where $y<0$) with the exception that the first coordinate of $\Phi_{O}$ changes its sign. In \cite{hom}, the author obtains precise asymptotic expansions for the local map $\Sigma^{in} \rightarrow \Sigma^{out}$. In the present article, we omit high order terms because they are not needed to our purposes.

\subsection{Transition Maps and First Return Maps}
\label{Transition Maps}
Let $j\in \{+,-\}$. 
By the \emph{Tubular Flow Theorem} \cite{PM}, solutions starting near $\Sigma_j^{out} \cap W^u_{loc}(O)$ follow one of the connections in $\Gamma$. We may then define the transition map $\Psi: \Sigma^{out} \rightarrow \Sigma^{in}$ by flow box fashion and the return map to $\Sigma^{in}$: $$R=\Psi \circ \Phi_O: \Sigma^{in} \rightarrow \Sigma^{in}.$$  Hereafter, we concentrate our attention on initial conditions that do not escape from $N_\Gamma$; otherwise take a smaller subset in $ \Sigma^{in}$ where the returm map is well defined. The explicit expression for the return map is highly nonlinear since the distortion near the hyperbolic saddle-foci is tremendous. For our purposes, it is not needed.

\section{Local Geometry}
\label{global}
The coordinates and notations of \S \ref{localdyn} will be used to study the geometry of the local dynamics near the saddle-focus. This is the main goal of the present section but first we introduce the concept of a \emph{segment} on $\Sigma^{in}$.

\begin{definition}
Let $j\in \{-,+\}$. A \emph{ segment }$\beta $:
\begin{enumerate}
\item \label{segment w} 
{\em on} $\Sigma^{in}_j$  is a smooth regular parametrized curve
$\beta :[0,1]\rightarrow \Sigma^{in}_j$ that meets $W^{s}_{loc}(O )$ at the point $\beta (0)$ and such that, writing $\beta (s)=(x(s),y(s))$, both $x$ and $y$ are monotonic and bounded functions of $s$.
\end{enumerate}
\end{definition}

The definition of \emph{segment} may be relaxed: the components do not need to be monotonic for all $s \in [0,1]$. We use the assumption of monotonicity to simplify the arguments.

\begin{definition}Let $a \in \RR$, $D$ be a disc centered at $p\in\RR^2$ and $\ell$ a line passing through $p$.
\begin{enumerate}
\item A \emph{spiral} on $D$ around the point $p$ is a smooth curve
$$\alpha :[a, +\infty[ \rightarrow D,$$ satisfying $\dpt \lim_{s\to +\infty }\alpha (s)=p$ and such that if
$\alpha (s)=(r(s),\phi(s))$ is its expression in polar coordinates around $p$ then:
\begin{enumerate}
\item the map $r$ is bounded by two monotonically decreasing maps converging to zero as $s \rightarrow +\infty$;
\item the map $\phi$ is monotonic for some unbounded subinterval of $[a,+\infty[$ and
\item $\lim_{s\to +\infty}|\phi(s)|=+\infty$.
\end{enumerate}
\item A \emph{double spiral} on $D$ around the point $p$ is the union of two spirals accumulating on $p$ and a curve connecting the other end points. 
\item Given a \emph{spiral} $\alpha$ on $D$ around the point $p$, an \emph{half circle} bounded by $\ell$ is a connected component of $\alpha \backslash \ell$.
\end{enumerate}

\end{definition}

 The next result characterizes the local dynamics near the saddle-focus.

\begin{lemma}
\label{Lemma3}
Let $j \in \{+,-\}$. A \emph{ segment }$\beta $ on   $\Sigma^{in}_j$ is mapped by $\Phi_O$ into a spiral on $\Sigma^{out}_j$  accumulating on the point defined by $\Sigma^{out}_j \cap W^u_{loc}(O)$.

\end{lemma}

\begin{proof}
The proof will be done for $j=+$. The other case is analogous. 
Let $\beta$ be a segment on $\Sigma_+^{in}$. Write $\beta(s) = (x^\star(s), y^\star(s)) \in \Sigma_+^{in}$, where:
\begin{itemize}
\item  $s \in [0,1]$,
\item $y^\star$ is an increasing map as function of $s$ and 
\item $\lim_{s \rightarrow 0^+}x^\star(s)=0$.
\end{itemize}

The function $\Phi_O$ maps the segment $\beta \subset \Sigma^{in}_+$ into  the curve defined by:
$$
\Phi_O(\beta(s))= \Phi_O[x^\star(s), y^\star(s)]= \left[(y^\star(s))^\delta,    -\frac{\alpha}{E}\ln
|y^\star(s)|+x^\star(s) \right ] = (r^\star(s), \phi^\star(s)).
$$
The map $\Phi_O \circ \beta$ is a spiral on $\Sigma^{out}_+$  accumulating on the point defined by $\Sigma^{out}_+ \cap W^u_{loc}(O)$
because $r(s)$ and $\phi(s)$ are monotonic (see Remark \ref{Rem1}) and
$$
\lim_{s \rightarrow 0^+} (y^\star(s))^{\delta}=0 \quad \text{and} \quad \lim_{s \rightarrow 0^+} \left|-\frac{\alpha}{E}\ln
y^\star(s)+x^\star(s)\right|=+\infty.
$$

\end{proof}

\begin{remark}
\label{Rem1}
Let $j \in \{+,-\}$.
 The coordinates $(r,\phi) \in \Sigma^{out}_j$ may be chosen so as to make the map $\phi^\star$ increasing or decreasing, according to our convenience.
From now on, we omit the dependence of $x^\star, y^\star$ on $s$ to simplify the notation.
\end{remark}

\section{Proof of Theorem \ref{Main}\textbf{(c)}} 
\label{Sec6}

In this section we put together the information about the return map. First note that if $A\subset M$, we denote by $\overline{A}$ and $int(A)$ its topological closure and topological interior, respectively. In what follows we remove the point $\beta(0)$ from the $graph(\beta)$ because the return map is not defined at this point.
From now on, let us fix $N_\Gamma$, a small neighbourhood of $\Gamma$. Figure \ref{neigh_vw} illustrates the main idea of the  following proof.

\subsection{The return map}
The next result shows that there are infinitely many points in $graph(\beta)\subset \Sigma_+^{in}$ which are mapped under $R$ into $W^s_{loc}(O)$ and that separate segments of initial conditions that follow the different connections $\gamma_1$ and $\gamma_2$. 
\begin{lemma}
\label{first lemma}
Let $\mathcal{R}$ be a rectangle in $\Sigma^{in}$ centered at one point of $\Gamma\cap \Sigma^{in}$. For any segment $\beta:(0,1] \rightarrow \mathcal{R}\cap \Sigma^{in}_+$, there is a family of intervals of the type $\Delta_i=[a_i, a_{i+1}]$ such that for all $i \in \mathbf{N}$, we have: 
\begin{enumerate}
\item $R\circ \beta (a_i)\in W^s_{loc}(O)$;
\item $R\circ \beta([a_i, a_{i+1}])$ is an half-circle in $\Sigma^{in}_+$ bounded by $W^s_{loc}(O)$;
\item $R\circ \beta([a_{i+1}, a_{i+2}])$ is an half-circle in $\Sigma^{in}_-$ bounded by $W^s_{loc}(O)$.
\end{enumerate}
\end{lemma}

\begin{proof}
As already said, we concentrate our attention on initial conditions that do not escape from $N_\Gamma$. 
By Lemma \ref{Lemma3}, the image of $\beta\subset \mathcal{R} $ under $\Psi_O$ is a spiral accumulating on $\Sigma_+^{out} \cap W^u_{loc}(O)$. In its turn, this spiral is mapped by $\Psi$ into another spiral in  $\Sigma^{in}$ accumulating on one point of $\Gamma\cap \Sigma^{in}$. 
The curve $W^s_{loc}(O)$ cuts transversely this spiral into infinitely many points. Let $s=a_i$ the points for which $R\circ \beta(a_i)\in W^s_{loc}(O)$. By construction, it is easy to see that if $s\in (a_i, a_{i+1})$ either $R\circ \beta(s)\in \Sigma_+^{in}$ or $R\circ \beta(s)\in \Sigma_-^{in}$. Suppose, without loss of generality, that the first case holds. Then, by continuity of $R$, for $s\in (a_{i+1}, a_{i+2})$ we get $R\circ \beta(s)\in \Sigma_-^{in}$. 
\end{proof}

\begin{lemma}
\label{second lemma}
Let $\mathcal{R}$ be a rectangle in $\Sigma^{in}$ centered at one point of $\Gamma\cap \Sigma^{in}$ and let $\Delta_i$ be as in Lemma \ref{first lemma}. 
Then for sufficiently large $i \in  \mathbf{N}$, there are $[b_{i,j}, b_{i, j+1}]$  such that for all $j \in \mathbf{N}$, we have: 
\begin{enumerate}
\item  $R^2\circ \beta(b_{i,j})\in W^s_{loc}(O)$.
\item $ [b_{i,j}, b_{i, j+1}]\subset \Delta_i$
\item $R^2\circ \beta([b_{i, j}, b_{i, j +1}])$ is an half-circle in $\Sigma^{in}_+$ bounded by $W^s_{loc}(O)$;
\item $R^2 \circ \beta([b_{i, j+1}, b_{i, j+2}])$ is an half-circle in $\Sigma^{in}_-$ bounded by $W^s_{loc}(O)$.

\end{enumerate}

\end{lemma}

\begin{proof}
Our starting point is the half-circle $R\circ\beta(int(\Delta_i))\subset \Sigma_+^{in}$, $i \in \mathbf{N}$, bounded by $W^s_{loc}(O)$. Since the image under $R$ of each half circle can be seen as the image of two connected segments, using Lemma \ref{Lemma3} the set  $\Pi_O\circ R\circ\beta(int(\Delta_i))\subset \Sigma_+^{out}$ is a double spiral accumulating on $W^u_{loc}\cap \Sigma_+^{out}$, which is mapped under $\Psi$ into a double spiral accumulating on $\Gamma\cap \Sigma^{in}$. The line $W^s_{loc}(O)$ intersects this double spiral infinitely many times. Let $s=\beta_{j,i}$ the sequence of points such that  $R^2\circ \beta(b_{i,j})\in W^s_{loc}(O)$. For a given $i\in \NN$, the arguments used before may be used to conclude that there exists a sequence $(b_{i,j})_j$ such that $R^2\circ \beta([b_{i, j}, b_{i, j +1}])$ is an half-circle in $\Sigma^{in}_+$ bounded by $W^s_{loc}(O)$ and $R^2\circ\beta([b_{i, j+1}, b_{i, j+2}])$ is an half-circle in $\Sigma^{in}_-$ bounded by $W^s_{loc}(O)$.
\end{proof}
\subsection{Finite switching}
 In the previous sections  we have proved that a segment  cutting transversely the stable manifold of $O$ contains subsegments $\Delta_i$ that are mapped into new segments cutting transversely the stable manifold of $O$. Starting with a segment $\beta$ on $\Sigma_+^{in}$, we may obtain, recursively, nested compact subsets containing initial conditions that follow any prescribed sequence of connections. 
\begin{definition}
Let $k, l \in \mathbf{N}$. We say that the path $\sigma^k=\left(\gamma_{\sigma(1)}, \ldots, \gamma_{\sigma(k)}\right)$ of order $k$ on the homoclinic network $\Gamma$ is inside the path $\sigma^{k+l}=\left(\gamma_{\omega(1)}, \ldots, \gamma_{\omega(k+l)}\right)$ of order $k+l$ if $\sigma(i)=\omega(i)$ for all $i \in \{1,\ldots, k\}$. We denote this relation by $\sigma^k \prec \sigma^{k+l}$.
\end{definition}

\begin{proposition}
\label{theoremFiniteSwitching}
There is finite switching near the network $\Gamma$ defined by a vector field satisfying \textbf{(H1)}--\textbf{(H3)}.
\end{proposition}

\begin{proof}
Given a path of order $k\in \mathbf{N}$, $\sigma^k= \left(\gamma_{\sigma(1)},\ldots, \gamma_{\sigma(k)}\right)\in \{\gamma_1, \gamma_2\}^k$, we want to find trajectories that follow it. Let us fix a segment $\Delta_{i}$ given by Lemma \ref{first lemma} and set that all initial conditions in $\Delta_{i}\setminus W_{loc}^s(O)$ follow the connection $\gamma_{\sigma(1)}$. Take a closed subset $\mathcal{A}_{i}$ of $\Delta_{i}$. 
By construction, all  initial conditions starting in $\mathcal{A}_{i}$ follow the connection $\gamma_{\sigma(1)}$.
The set $R\circ \beta(\Delta_i)$ is an half circle cutting transversely $W^s_{loc}(O)\cap \Sigma^{in}$ infinitely many times. By Lemma \ref{second lemma}, one can obtain again sequences of points in $\Delta_{i,j}$, where a similar result to that in Lemma \ref{first lemma} can be stated for $R^2$ instead of $R$. Take a closed subset $\mathcal{A}_{j,i}$ of $\Delta_{j,i}$. 
By construction, all  initial conditions starting in $\mathcal{A}_{j,i}$ follow the path $(\gamma_{\sigma(1)},\gamma_{\sigma(2)})$.
A recursive argument allows the construction of a compact set $\mathcal{A}_{{\sigma(1)},{\sigma(2)},\ldots, {\sigma(k)}}$ of initial conditions whose trajectories follow $\sigma^{k}$.
\end{proof}

\subsection{Proof of Theorem \ref{Main}(c)}
\label{Infinite switching}

\label{Infinite2}
We first need to introduce some extra terminology.
Given a path of order $k\in \mathbf{N}$,
$$\sigma^k= \left(\gamma_{\sigma(1)},\ldots,\gamma_{\sigma(k)}\right) \in \{\gamma_1, \gamma_2\}^k,$$
we denote by $\mathcal{Q}(\sigma^k)$ the compact set $\mathcal{A}_{{\sigma(1)}{\sigma(2)}\ldots {\sigma(k)}}$ obtained in the proof of Proposition \ref{theoremFiniteSwitching} and we say that $\mathcal{Q}(\sigma^k)$ is an admissible set. Recall that all points in $\mathcal{Q}(\sigma^k)$ correspond to inifinitely many solutions following $\sigma^{k}$.

\begin{remark}
\label{admissible}
By the construction in the proof of Proposition \ref{theoremFiniteSwitching}, if $\sigma^k \prec \sigma^{k+l}$ one can get admissible sets such that $\mathcal{Q}(\sigma^k) \supset \mathcal{Q}(\sigma^{k+l})$.
\end{remark}

\begin{proof}
Fix an infinite path $\sigma^{\infty}=(\gamma_{\sigma(j)})_{j \in \mathbf{N}}$, with $\sigma:\mathbf{N}\to\{1 ,2\}$. For each $k \in \mathbf{N}$ define the finite path $\sigma^{k}=(\gamma_{\sigma(j)})_{j \in \ \left\{1,\ldots,k\right\} }$. Taking into account Remark \ref{admissible} it follows that there exists an infinite sequence of admissible sets $\{\mathcal{Q}(\sigma^{k})\}_{k\in\mathbf{N}}$ such that $\mathcal{Q}(\sigma^k) \supset \mathcal{Q}(\sigma^{k+1})$ for all $k\in\mathbf{N}$. Since the sequence of compact sets $\{\mathcal{Q}(\sigma^{k})\}_{k\in\mathbf{N}}$ is nested, $\mathcal{B}=\bigcap_{k=1}^\infty \mathcal{Q}(\sigma^k)\neq \emptyset$.
Any initial condition in $\mathcal{B}$ gives a trajectory which follows $\sigma^{\infty}$. The different solutions are distinguished by the number of revolutions around the isolating block of $O$ \cite{Rodrigues}. Again by construction we find trajectories realising the required switching arbitrarily close to $\Gamma$. 
\end{proof}

Based on Theorem \ref{Main}, the answer to the question \textbf{(Q1)} is \emph{yes}.  At this point, other questions arise, namely: \emph{are there other examples of heteroclinic networks where infinite switching holds without suspended horseshoes emerging in its neighbourhood?}

\end{document}